\documentclass[12pt, reqno]{amsart}

\usepackage{mathrsfs}
\usepackage{amsfonts}
\usepackage[centertags]{amsmath}
\usepackage{amssymb,array}
\usepackage{amsthm}
\usepackage{stmaryrd}
\usepackage{graphicx}
\usepackage{caption}
\usepackage{ytableau}
\usepackage{enumerate,enumitem}
\usepackage{MnSymbol}
\SetLabelAlign{center}{\hfill #1\hfill}
\usepackage[textwidth=16cm, hmarginratio=1:1]{geometry}
\usepackage{tikz-cd, tikz}
\usepackage{rotating}
\usepackage[all,cmtip]{xy}
\usepackage[titletoc, title]{appendix}
\usepackage{amscd}
\usepackage[colorlinks]{hyperref}
\usepackage{float}
\usepackage{mathtools}
\usepackage{dsfont} 
\hypersetup{
bookmarksnumbered,
pdfstartview={FitH},
breaklinks=true,
linkcolor=blue,
urlcolor=blue,
citecolor=blue,
bookmarksdepth=2
}

\theoremstyle{plain}
   
   \newtheorem{theorem}{Theorem}[section]
   \newtheorem{proposition}[theorem]{Proposition}
   
   \newtheorem{lemma}[theorem]{Lemma}
   
   \newtheorem{conjecture}[theorem]{Conjecture}
   
\theoremstyle{definition}
   \newtheorem{definition}[theorem]{Definition}
   
   \newtheorem{example}[theorem]{Example}

   \newtheorem{remark}[theorem]{Remark}

\numberwithin{equation}{section}

\newcommand{\CC}{{\mathbb {C}}}

\newcommand{\ZZ}{{\mathbb {Z}}}

\newcommand{\ch}{{\operatorname{ch}}}

\newcommand{\SSYT}{{\rm SSYT}}

\DeclareMathOperator{\Gr}{Gr}

\newcommand\scalemath[2]{\scalebox{#1}{\mbox{\ensuremath{\displaystyle #2}}}}

\newlength{\mysizetiny}
\newlength{\mysizesmall}
\newlength{\mysize}
\newlength{\mysizelarge}
\newenvironment{nouppercase}{%
	\renewcommand{\uppercasenonmath}[1]{}}{}

\begin{document}

\title[Classification of prime modules corresponding to 2-column tableaux]{Classification of prime modules of quantum affine algebras corresponding to 2-column tableaux} 

\author{Nick Early and Jian-Rong Li}
\address{Nick Early, Institute for Advanced Study, Princeton, NJ.}
\email{\href{mailto:earlnick@ias.edu}{earlnick@ias.edu}}   

\address{Jian-Rong Li, Faculty of Mathematics, University of Vienna, Oskar-Morgenstern-Platz 1, 1090 Vienna, Austria.} 
\email{jianrong.li@univie.ac.at}

\date{}
\begin{nouppercase}
\maketitle
\end{nouppercase}

{\centering\footnotesize \textit{Dedicated to Professor Vyjayanthi Chari on the occasion of her 65th birthday}\par}

\begin{abstract} 
Finite dimensional simple modules of quantum affine algebras of type A correspond to semistandard Young tableaux of rectangular shapes. In this paper, we classify all prime modules corresponding to 2-column semistandard Young tableaux, up to a conjectural property.
Moreover, we give a conjectural sufficient condition for a module corresponding to a tableau with more than two columns to be prime. 
\end{abstract}

\section{Introduction}

A quantum affine algebra $U_q(\widehat{\mathfrak{g}})$ is a Hopf algebra that is a $q$-deformation of the universal enveloping algebra of an affine Lie algebra $\widehat{\mathfrak{g}}$, see \cite{CP94}. Chari and Pressley classified finite dimensional simple $U_q(\widehat{\mathfrak{g}})$-modules \cite{CP91, CP94}: every finite dimensional simple $U_q(\widehat{\mathfrak{g}})$-module corresponds to an $I$-tuple of polynomials $(p_i(u))_{i \in I}$ called Drinfeld polynomials, where $I$ is the set of vertices of the Dynkin diagram of $\mathfrak{g}$, $p_i(u) \in \CC[u]$ and each $p_i(u)$ has constant term $1$. Each $I$-tuple of Drinfeld polynomials $(p_i(u))_{i \in I}$ corresponds to a dominant monomial in formal variables $Y_{i,a}$, $i \in I$, $a \in \mathbb{C}^{\times}$, where dominant means the exponents appearing in the monomial are all non-negative. The finite dimensional simple module corresponding to a dominant monomial $M$ is denoted by $L(M)$. 

A simple module $L(M)$ is called prime if $L(M)$ cannot be decomposed as the tensor product of two non-trivial simple modules. That is, $L(M) \not\cong L(M') \otimes L(M'')$ for any non-trivial modules $L(M')$, $L(M'')$. Prime modules of $U_q(\widehat{\mathfrak{sl}_2})$ have been classified by Chari and Pressley \cite{CP97}. They proved that all prime modules of $U_q(\widehat{\mathfrak{sl}_2})$ are Kirillov-Reshetikhin modules. Kirillov-Reshetikhin modules are simple $U_q(\widehat{\mathfrak{g}})$-modules which correspond to dominant monomials of the form $Y_{i,s}Y_{i,s+2} \cdots Y_{i,s+2r}$ (when $\mathfrak{g}$ is of simply-laced type), where $i \in I$, $r \in \mathbb{Z}_{\ge 0}$. In general, it is a difficult problem to classify prime modules of $U_q(\widehat{\mathfrak{g}})$. Prime modules have been studied intensively in the literature, see for example, \cite{BCM18, BC22, CMY13, CP96a, CP97, EL23, FM00, HL10, MouS22}.  

We denote by $\Gr(k,n)$ the Grassmannian of $k$-planes in $\mathbb{C}^n$ and $\CC[\Gr(k,n)]$ its homogeneous coordinate ring. It was shown by Scott \cite{Sco06} that the ring $\CC[\Gr(k,n)]$ has a cluster algebra structure. The algebra  $\CC[\Gr(k,n)]$ is called a Grassmannian cluster algebra. 

Hernandez and Leclerc \cite[Section 13]{HL10} proved that the Grothendieck ring $K_0(\mathcal{C}_{\ell})$ of certain subcategory $\mathcal{C}_{\ell}$ (see Section \ref{subsec:HL category and cluster algebra}) of the category of finite dimensional $U_q(\widehat{\mathfrak{sl}_k})$-modules is isomorphic to a certain quotient $\CC[\Gr(k,n,\sim)]$ (certain frozen variables are sent to $1$, see Section \ref{subsec:Grassmannian cluster algebras and semistandard Young tableaux}) of the Grassmannian cluster algebra $\CC[\Gr(k,n)]$, where $n = k+\ell+1$. Denote by $\SSYT(k,[n])$ the set of semistandard Young tableaux of rectangular shapes with $k$ rows and with entries in $[n]=\{1,\ldots, n\}$. 
It was shown in \cite{CDFL} that there is a one to one correspondence between the elements in the dual canonical basis of $K_0(\mathcal{C}_{\ell})$ (resp. $\CC[\Gr(k,n,\sim)]$) and semistandard Young tableaux in $\SSYT(k,[n],\sim)$ (we also call an equivalence class in $\SSYT(k,[n],\sim)$ a tableau), where $\SSYT(k,[n],\sim)$ is a certain quotient of $\SSYT(k,[n])$, see Section \ref{subsec:Grassmannian cluster algebras and semistandard Young tableaux}. We say that a tableau is prime if the corresponding module is prime. Therefore, classification of prime modules in $\mathcal{C}_{\ell}$ is equivalent to classification of prime tableaux in $\SSYT(k,[n],\sim)$.

Recently, cluster variables of $\CC[\Gr(k,n)]$ corresponding to 2-column tableaux have been studied in \cite[Section 4]{BBGL} and \cite{LeY23}. The set of cluster variables corresponding to 2-column tableaux is a subset of the set of 
prime elements in the dual canonical basis of $\CC[\Gr(k,n)]$ corresponding to 2-column tableaux.  

In this paper, we classify all prime modules corresponding to 2-column semistandard Young tableaux, up to a conjectural property in Conjecture \ref{conjecture: tensor product is simple implies that I J weakly separated}. We prove the following property in Lemma \ref{lem:semistandard tableau and noncrossing tuple}: for every tableau $T \in {\rm SSYT}(k, [n])$ which has $m$ columns, there is a unique unordered $m$-tuple $(S_1, \ldots, S_m)$ of one-column tableaux which are pairwise noncrossing such that $T = S_1 \cup \cdots \cup S_m$. Let $L(M)$ be a simple $U_q(\widehat{\mathfrak{sl}_k})$-module such that the corresponding tableau $T_M$ is of 2-column. By Lemma \ref{lem:semistandard tableau and noncrossing tuple}, there is a unique pair $T_1, T_2$ of one-column tableaux $T_1, T_2$ such that $T_1, T_2$ are noncrossing and $T_M = T_1 \cup T_2$. Under the assumption that Conjecture \ref{conjecture: tensor product is simple implies that I J weakly separated} (For two $k$-element subsets $J, J'$ of $[n]$, if $L(M_{J}) \otimes L(M_{J'})$ is simple, then $J, J'$ are weakly separated) is true, we prove that $L(M)$ is prime if and only if $T_1, T_2$ are not weakly separated, see Theorem \ref{thm: Uqslkhat modules corresponding to 2 column tableaux}. We also count the number of prime modules corresponding to $2$-column tableaux: for $k \le n/2$, the number of $2$-column prime tableaux is $a_{k,n,2}-b_{k,n}$, where $a_{k,n,m}=\prod_{i=1}^k \prod_{j=1}^m \frac{n-i+j}{k+m-i-j+1}$, $b_{k,n} = \binom{n}{k} + \sum_{j=1}^{k} j \genfrac{(}{)}{0pt}{}{n}{k-j, 2j, n-k-j}$, and $\genfrac{(}{)}{0pt}{}{n}{a,b,c}=\frac{n!}{a!b!c!}$, see Proposition \ref{prop:number of 2-column prime tableaux}.

Moreover, we give a conjectural sufficient condition for a simple module corresponding to a tableau with more than 2 columns to be prime. Let $T \in {\rm SSYT}(k, [n])$ and let $(S_1, \ldots, S_m)$ be the unique unordered $m$-tuple $(S_1, \ldots, S_m)$ of one-column tableaux which are pairwise noncrossing such that $T = S_1 \cup \cdots \cup S_m$. We conjecture that if for every $i \ne j$, $S_i, S_j$ are not weakly separated, then $T$ is prime, see Conjecture \ref{conj:pairwise noncrossing and not weakly separated collections induce prime modules}. 

The paper is organized as follows. In Section \ref{sec:preliminaries}, we recall results of quantum affine algebras, Hernandez--Leclerc's category $\mathcal{C}_{\ell}$, and Grassmannian cluster algebras. In Section \ref{sec:prime modules corresponding to 2-column tableaux}, we classify all prime modules corresponding to 2-column semistandard Young tableaux. In Section \ref{sec:2-column prime tableaux in Gr48 and Gr510}, we list prime modules corresponding to 2-column prime tableaux for $\CC[\Gr(4,8)]$ and $\CC[\Gr(5,10)]$. In Section \ref{sec:Prime modules corresponding to tableaux with two or more columns}, we give a conjectural sufficient condition for a module corresponding to a tableau with more than two columns to be prime. 

\subsection*{Acknowledgements}
The authors would like to thank Fedor Petrov for his help of proving Proposition \ref{prop:number of 2-column prime tableaux}. The authors would like to thank the reviewers for their insightful comments and valuable suggestions, which have improved the manuscript. JRL is supported by the Austrian Science Fund (FWF): P-34602, Grant DOI: 10.55776/P34602, and PAT 9039323, Grant-DOI 10.55776/PAT9039323. This research received funding from the European Research Council (ERC) under the European Union’s Horizon 2020 research and innovation programme (grant agreement No 725110), Novel structures in scattering amplitudes. N.E. was funded by the European Union (ERC, UNIVERSE PLUS, 101118787). \begin{tiny}
Views and opinions expressed are however those of the author(s) only and do not necessarily reflect those of the European Union or the European Research Council Executive Agency. Neither the European Union nor the granting authority can be held responsible for them.
\end{tiny}  

\section{Preliminaries} \label{sec:preliminaries}
In this section, we recall results of quantum affine algebras \cite{CP94, FR98}, Hernandez-Leclerc's category $\mathcal{C}_{\ell}$ \cite{HL10}, and Grassmannian cluster algebras \cite{Sco06, CDFL}. 

\subsection{Quantum affine algebras} \label{subsec:quantum affine algebras}

Let $\mathfrak{g}$ be a simple finite-dimensional Lie algebra and $I$ the set of vertices of the Dynkin diagram of $\mathfrak{g}$. The quantum affine algebra $U_q(\widehat{\mathfrak{g}})$ is a Hopf algebra that is a $q$-deformation of the universal enveloping algebra of $\widehat{\mathfrak{g}}$ \cite{Drin87, Jim85}. In this paper, we take $\mathfrak{g}$ to be of type $A$, i.e., $\mathfrak{g}=\mathfrak{sl}_k$ for $k \in \ZZ_{\ge 1}$, and take $q$ to be a non-zero complex number which is not a root of unity. 

Denote by $\mathcal{P}$ the free abelian group generated by formal variables $Y_{i, a}^{\pm 1}$, $i \in I$, $a \in \CC^*$, and denote by $\mathcal{P}^+$ the submonoid of $\mathcal{P}$ generated by $Y_{i, a}$, $i \in I$, $a \in \CC^*$. Let $\mathcal{C}$ denote the monoidal category of finite-dimensional representations of the quantum affine algebra $U_q(\widehat{\mathfrak{g}})$. Any finite dimensional simple object in $\mathcal{C}$ is a highest $l$-weight module with a highest $l$-weight $M \in \mathcal{P}^+$, denoted by $L(M)$ (see \cite{CP95a}). The elements in $\mathcal{P}^+$ are called dominant monomials.  

Frenkel and Reshetikhin \cite{FR98} introduced the $q$-character map which is an injective ring morphism $\chi_q$ from the Grothendieck ring of $\mathcal{C}$ to $\mathbb{Z}\mathcal{P} = \mathbb{Z}[Y_{i, a}^{\pm 1}]_{i\in I, a\in \mathbb{C}^{\times}}$. For a $U_q(\widehat{\mathfrak{g}})$-module $V$, $\chi_q(V)$ encodes the decomposition of $V$ into common generalized eigenspaces for the action of a large commutative subalgebra of $U_q(\widehat{\mathfrak{g}})$ (the loop-Cartan subalgebra). These generalized eigenspaces are called $l$-weight spaces and generalized eigenvalues are called $l$-weights. One can identify $l$-weights with monomials in $\mathcal{P}$ \cite{FR98}. Then
the $q$-character of a $U_q(\widehat{\mathfrak{g}})$-module $V$ is given by (see \cite{FR98})
\begin{align*}
\chi_q(V) = \sum_{  m \in \mathcal{P}} \dim(V_{m}) m \in \mathbb{Z}\mathcal{P},
\end{align*}
where $V_{m}$ is the $l$-weight space with $l$-weight $m$. 

For $i \in I$, $a \in \mathbb{C}^{\times}$, $k \in \ZZ_{\ge 1}$, the modules
\begin{align*}
X_{i,a}^{(k)} := L(Y_{i,a} Y_{i,aq^2} \cdots Y_{i,aq^{2k-2}})
\end{align*}
are called Kirillov-Reshetikhin modules. The modules $X_{i,a}^{(1)} = L(Y_{i,a})$ are called fundamental modules.

\subsection{\texorpdfstring{Hernandez-Leclerc's category $\mathcal{C}_{\ell}$}{Hernandez-Leclerc's category C\_ell}} \label{subsec:HL category and cluster algebra}

We recall the definition of Hernandez-Leclerc's category $\mathcal{C}_{\ell}$ \cite{HL10} in the case when $\mathfrak{g}=\mathfrak{sl}_k$. 

For integers $a \le b$, we denote $[a,b] = \{i: a \le i \le b\}$ and $[a] = \{i: 1 \le i \le a\}$. Let $\mathfrak{g}=\mathfrak{sl}_k$ over $\CC$ and let $\mathcal{C}$ be the category of finite-dimensional $U_q(\widehat{\mathfrak{g}})$-modules. In \cite{HL10}, Hernandez and Leclerc introduced a full subcategory $\mathcal{C}_{\ell} = \mathcal{C}_{\ell}^{\mathfrak{g}}$ of $\mathcal{C}$ for every $\ell \in \mathbb{Z}_{\geq 0}$. Let $I=[k-1]$ be the set of vertices of the Dynkin diagram of $\mathfrak{g}$. We fix $a \in \CC^*$ and denote $Y_{i,s} = Y_{i,aq^s}$, $i \in I$, $s \in \ZZ$. For $\ell \in \ZZ_{\ge 0}$, denote by $\mathcal{P}_\ell$ the subgroup of $\mathcal{P}$ generated by $Y_{i,\xi(i)-2r}^{\pm 1}$, $i \in I$, $r \in [0, \ell]$, where $\xi: I \to \ZZ$ is a height function defined by $\xi(i)=i-2$. Denote by $\mathcal{P}^+_\ell$ the submonoid of $\mathcal{P}^+$ generated by $Y_{i,\xi(i)-2r}$, $i \in I$, $r \in [0, \ell]$. An object $V$ in $\mathcal{C}_{\ell}$ is a finite-dimensional $U_q(\widehat{\mathfrak{g}})$-module which satisfies the condition: for every composition factor $S$ of $V$, the highest $l$-weight of $S$ is a monomial in $\mathcal{P}^+_\ell$, \cite{HL10}. Simple modules in $\mathcal{C}_{\ell}$ are of the form $L(M)$ (see \cite{CP94}, \cite{HL10}), where $M \in \mathcal{P}_{\ell}^+$.

For every $\ell \ge 0$, Hernandez and Leclerc constructed a cluster algebra for $\mathcal{C}_{\ell}$ of $U_q(\widehat{\mathfrak{g}})$ \cite{HL10}. The cluster algebra for $\mathcal{C}_{\ell}$ of $U_q(\widehat{\mathfrak{sl}_k})$ is isomorphic to the cluster algebra for a certain quotient $\CC[\Gr(k,n,\sim)]$ (see Section \ref{subsec:Grassmannian cluster algebras and semistandard Young tableaux}) of the Grassmannian cluster algebra $\CC[\Gr(k,n)]$ \cite{CDFL, HL10, Sco06}, $n=k+\ell+1$. 
	
\subsection{Grassmannian cluster algebras and semistandard Young tableaux}\label{subsec:Grassmannian cluster algebras and semistandard Young tableaux}

For $k \le n$, the Grassmannian $\Gr(k,n)$ is the set of $k$-dimensional subspaces in an $n$-dimensional vector space. In this paper, we denote by $\Gr(k,n)$ (the affine cone over) the Grassmannian of $k$-dimensional subspaces in $\CC^n$, and denote by $\CC[\Gr(k,n)]$ its coordinate ring. This algebra is generated by Pl\"{u}cker coordinates 
\begin{align*}
P_{i_1, \ldots, i_{k}}, \quad 1 \leq i_1 < \cdots < i_{k} \leq n.
\end{align*}

It was shown by Scott \cite{Sco06} that the ring $\CC[\Gr(k,n)]$ has a cluster algebra structure. Define $\CC[\Gr(k,n,\sim)]$ to be the quotient of $\CC[\Gr(k,n)]$ by the ideal generated by $P_{i, \ldots, i+k-1}-1$, $i \in [n-k+1]$. In \cite{CDFL}, it is shown that the elements in the dual canonical basis of $\CC[\Gr(k,n,\sim)]$ are in bijection with semistandard Young tableaux in ${\rm SSYT}(k, [n],\sim)$, where ${\rm SSYT}(k, [n],\sim)$ is defined in the following paragraphs. 

A semistandard Young tableau is a Young tableau with weakly increasing rows and strictly increasing columns. For $k,n \in \ZZ_{\ge 1}$, we denote by ${\rm SSYT}(k, [n])$ the set of rectangular semistandard Young tableaux with $k$ rows and with entries in $[n]$ (with arbitrarly many columns). The empty tableau is denoted by $\mathds{1}$. 

For $S,T \in {\rm SSYT}(k, [n])$, let $S \cup T$ be the row-increasing tableau whose $i$th row is the union of the $i$th rows of $S$ and $T$ (as multisets), \cite{CDFL}. It is shown in Section 3 in \cite{CDFL} that $S \cup T$ is semistandard for any pair of semistandard tableaux $S, T$.  

We call $S$ a factor of $T$, and write $S \subset T$, if the $i$th row of $S$ is contained in that of $T$ (as multisets), for $i \in [k]$. In this case, we define $\frac{T}{S}=S^{-1}T=TS^{-1}$ to be the row-increasing tableau whose $i$th row is obtained by removing that of $S$ from that of $T$ (as multisets), for $i \in [k]$. 

A tableau $T \in {\rm SSYT}(k, [n])$ is trivial if each entry of $T$ is one less than the entry below it. For any $T \in {\rm SSYT}(k, [n])$, we  denote by $T_{\text{red}} \subset T$ the semistandard tableau obtained by removing a maximal trivial factor from $T$. For a trivial $T$, one has $T_{\text{red}} = \mathds{1}$. 

Let ``$\sim$'' be the equivalence relation on $S, T \in {\rm SSYT}(k, [n])$ defined by: $S \sim T$ if and only if $S_{\text{red}} = T_{\text{red}}$. We denote by ${\rm SSYT}(k, [n],\sim)$ the set of $\sim$-equivalence classes. We also call $\sim$-equivalence classes in ${\rm SSYT}(k, [n],\sim)$ tableaux. For a tableau $T$ ($\sim$-equivalence class) in ${\rm SSYT}(k, [n],\sim)$, we choose the representative of the class as the unique tableau which has smallest number of columns and we say that the number of columns of $T$ is the number of columns of this representative.
 
The elements in the dual canonical basis of $\CC[\Gr(k,n,\sim)]$ are in bijection with simple modules in the category $\mathcal{C}_{\ell}$ of $U_q(\widehat{\mathfrak{sl}_k})$ in Section \ref{subsec:Grassmannian cluster algebras and semistandard Young tableaux}, see \cite{HL10, CDFL}.

A one-column tableau is called a fundamental tableau if its entries are $[i,i+k] \setminus \{r\}$ for $r \in \{i+1, \ldots, i+k-1\}$. 
Any tableau in $\SSYT(k,[n])$ is $\sim$-equivalent to a unique semistandard tableau whose columns are fundamental tableaux, see Lemma 3.13 in \cite{CDFL}. A semistandard tableau whose columns are fundamental tableaux is called a small gap tableau. 

We now recall the explicit formula of $\ch(T)$ in the dual canonical basis of $\CC[\Gr(k,n,\sim)]$ in \cite[Theorem 5.8]{CDFL}. For $m \in \ZZ_{\ge 1}$, denote by $S_m$ the symmetric group on $[m] = \{1, \ldots, m\}$. For that we need to first define $w_T \in S_m$, $P_{u; T'}$, $u \in S_m$, for every $T \in \SSYT(k, [n])$, where $T'$ is the unique small gap tableau which is $\sim$-equivalent to $T$, and $m$ is the number of columns of $T'$. 

Let ${\bf i} = (i_1 \leq i_2 \dots \leq i_m)$ be the entries in the first row of $T'$, and let $r_1,\dots,r_m$ be the elements such that the $a$th column of $T'$ has entries $[i_a,i_a+n] \setminus \{r_a\}$. Let ${\bf j} = (j_1 \leq j_2 \leq \dots \leq j_m)$ be the elements $r_1,\dots,r_m$ written in weakly increasing order. 

For a semistandard tableau $T$, denote $P_T = P_{T_1} \cdots P_{T_m}$, where $T_1, \ldots, T_m$ are columns of $T$. For $u \in S_m$, we define $P_{u;T'} \in \CC[\Gr(k,n)]$ as follows. Provided $j_a \in [i_{u(a)}, i_{u(a)}+k]$ for all $a \in [m]$, define the tableau $\alpha(u;T')$ to be the semistandard tableau whose columns have entries $[i_{u(a)}, i_{u(a)}+k] \setminus \{j_a\}$ for $a \in [m]$, and define $P_{u; T'} = P_{\alpha(u;T')} \in \CC[\Gr(k,n)]$ to be the corresponding standard monomial. On the other hand, if $j_a \notin [i_{u_a}, i_{u(a)}+k]$ for some $a$, then the tableau $\alpha(u;T')$ is {\sl undefined} and $P_{u ;T'} = 0$. 

There is a unique $u \in S_m$ which is of maximal length with the property that the sets 
\[
\{[i_{u(a)},i_{u(a)}+k] \setminus \{j_a \} \}_{a \in [m]}
\]
describe the columns of $T'$. This $u$ is denoted by $u = w_{T}$.

By \cite[Theorem 5.8]{CDFL}, the element $\ch(T)$ in the dual canonical basis of $\CC[\Gr(k,n,\sim)]$ is given by 
\begin{align}\label{eq:formula of ch(T)}
\ch(T) = \sum_{u \in S_m} (-1)^{\ell(uw_T)} p_{uw_0, w_Tw_0}(1) P_{u; T'},
\end{align}
where $p_{u,v}(q)$ is a Kazhdan-Lusztig polynomial \cite{KL79}.

\subsection{Relation between dominant monomials and tableaux} \label{subsec:dominant monomials and tableaux}
 
In Section \ref{subsec:HL category and cluster algebra}, we recalled Hernandez and Leclerc's category $\mathcal{C}_{\ell}$. It is shown in Theorem 3.17 in \cite{CDFL} that in the case of $\mathfrak{g}=\mathfrak{sl}_k$, the monoid $\mathcal{P}^+_\ell$ (we take the height function to be $\xi(i)=i-2$, $i \in [k-1]$) of dominant monomials is isomorphic to the monoid of semistandard Young tableaux $\SSYT(k, [n], \sim)$, $n = k+\ell+1$. The correspondence of dominant monomials and tableaux is induced by the following map sending variables $Y_{i,s}$ to fundamental tableaux:
\begin{equation}\label{eq:monomToTableaux} 
Y_{i,s} \mapsto T_{i,s}, 
\end{equation}
where $T_{i,s}$ is a one-column tableau consisting of entries $\frac{i-s}{2}, \frac{i-s}{2}+1, \ldots, \frac{i-s}{2}+k-i-1, \frac{i-s}{2}+k-i+1, \ldots, \frac{i-s}{2}+k$. We denote the monomial corresponding to a tableau $T$ by $M_T$ and denote the tableau in ${\rm SSYT}(k, [n],\sim)$ corresponding to a monomial $M$ by $T_M$. Note that by the definition of $\mathcal{C}_{\ell}$ and the choice of the height function $\xi(i)=i-2$, $i \in [k-1]$, the indices of $Y_{i, s}$ in the highest $l$-weight monomials of simple modules in $\mathcal{C}_{\ell}$ satisfy $i-s \pmod 2=0$. 

When computing the monomial corresponding to a given tableau, we first decompose the tableau into a union of fundamental tableaux. Then we send each fundamental tableau to the corresponding $Y_{i,s}$. For example, the tableaux $[[1,2,4,6],[3,5,7,8]]$ (each list is a column of the tableau), $[[1,3,5,7],[2,4,6,8]]$ correspond to the modules
\begin{align*}
L(Y_{2,-6}Y_{1,-3}Y_{3,-3}Y_{2,0}), \quad L(Y_{1,-7}Y_{2,-4}Y_{1,-5}Y_{3,-1}Y_{2,-2}Y_{3,1}),
\end{align*}
respectively.   

Recall that a simple $U_q(\widehat{\mathfrak{g}})$-module $L(M)$ is called prime if it is not isomorphic to $L(M') \otimes L(M'')$ for any non-trivial modules $L(M')$, $L(M'')$ \cite{CP97}. A simple $U_q(\widehat{\mathfrak{g}})$-module $L(M)$ is called real if $L(M) \otimes L(M)$ is still simple \cite{Lec}. We say that a tableau $T$ is prime (resp. real) if the corresponding $U_q(\widehat{\mathfrak{sl}_k})$-module $L(M_T)$ is prime (resp. real). The problem of classification of prime $U_q(\widehat{\mathfrak{sl}_k})$-modules in the category $\mathcal{C}_{\ell}$ ($\ell \ge 0$) is equivalent to the problem of classification of prime tableaux in $\SSYT(k, [n], \sim)$, $n = k + \ell + 1$, \cite{CDFL}. 

\subsection{Weakly separated property and noncrossing property} \label{subsec:weakly separated and noncrossing}

For two $k$-element subsets $I,J \subset [n]$, denote $I < J$ if $\max(I) < \min(J)$. 
\begin{definition} [{\cite{LZ98}}] \label{def:weakly separted k-subsets}
A pair of $k$-element subsets $I,J$ is said to be weakly separated if one of the following holds:
\begin{itemize}
\item $I\setminus J = I_1 \sqcup I_2$, $I_1 < J \setminus I <I_2$,
\item $J\setminus I = J_1 \sqcup J_2$, $J_1 < I \setminus J <J_2$. 
\end{itemize}
\end{definition}

\begin{definition}[{\cite{SSW17}}] \label{def:noncrossing pairs}
A pair $I=\{i_1< \ldots < i_k\}$, $J=\{j_1<\ldots<j_k\}$ of $k$-subsets of $[n]$ is said to be noncrossing if for each $1 \le a < b \le k$, either the pair $\{i_a, i_{a+1}, \ldots, i_b\}$, $\{j_a, j_{a+1}, \ldots, j_b\}$ is weakly separated, or $\{i_{a+1}, \ldots, i_{b-1}\} \ne \{j_{a+1}, \ldots, j_{b-1}\}$. 
\end{definition}

\section{Explicit Description of 2-column Prime Tableaux} \label{sec:prime modules corresponding to 2-column tableaux}

In this section, we prove that a $2$-column tableau is prime if and only if it is the union of two one-column tableaux which are noncrossing and not weakly separated. We also compute the number of $2$-column prime tableaux in $\SSYT(k, [n])$.

\subsection{Semistandard Young tableaux and noncrossing tuples}

We will prove that there is a one to one correspondence between semistandard Young tableaux and noncrossing tuples. This result could be seen as a tableaux analog of Theorem 1.2 in \cite{E2021}. First we consider the case of $k=2$. 
\begin{lemma} \label{lem:semistandard tableau and noncrossing tuple k=2}
For every tableau $T \in {\rm SSYT}(2, [n])$ which has $m$ columns, there is a unique unordered $m$-tuple $(S_1, \ldots, S_m)$ of one-column tableaux which are pairwise noncrossing such that $T = S_1 \cup \cdots \cup S_m$. 
\end{lemma}

\begin{proof}
First note that for $2$-row one column tableaux $\scalemath{0.6}{ \begin{ytableau}
a \\ b
\end{ytableau}}$, $\scalemath{0.6}{\begin{ytableau}
a' \\ b'
\end{ytableau}}$, they are noncrossing if and only if they are weakly separated. If $b=a+1$, then $\scalemath{0.6}{\begin{ytableau}
a \\ b
\end{ytableau}}$ is weakly separated with any $2$-row one-column tableau. Let $T$ be a $2$-row tableau and let $T'$ be the tableau obtained from $T$ by removing all factors of the form $\scalemath{0.6}{\begin{ytableau}
a \\ \scalemath{0.75}{ a+1 }
\end{ytableau}}$. Denote these frozen factors by $T''_1, \ldots, T''_t$. By Theorem 1.1 in \cite{CDFL}, $T'$ corresponds to a simple $U_q(\widehat{\mathfrak{sl}_2})$-module $L(M_T) = L(M_{T'})$. By Sections 4.8, 4.9, 4.11 in \cite{CP91}, every prime $U_q(\widehat{\mathfrak{sl}_2})$-module is a Kirillov-Reshetikhin module and every simple $U_q(\widehat{\mathfrak{sl}_2})$-module is decomposed as a tensor product of Kirillov-Reshetikhin modules (note that evaluation modules of $U_q(\widehat{\mathfrak{sl}_2})$ are Kirillov-Reshetikhin modules). Therefore 
\begin{align} \label{eq:decomposition of affine Uqsl2 modules}
\chi_q(L(M_T)) = \chi_q(L(M_1)) \cdots \chi_q(L(M_r))
\end{align}
for some Kirillov-Reshetikhin modules $L(M_1), \ldots, L(M_r)$. Every Kirillov-Reshetikhin module corresponds to a one-column tableau, see Section 3.3 in \cite{CDFL}. Let $T_{M_1}, \ldots, T_{M_r}$ be the one-column tableaux corresponding to $L(M_1), \ldots, L(M_r)$ respectively. By Equation (\ref{eq:decomposition of affine Uqsl2 modules}), 
we have that for any $i, j$, $L(M_i) \otimes L(M_j)$ is simple. Hence by Theorem 1.1 in \cite{LZ98}, $T_{M_i}$ and $T_{M_j}$ are weakly separated. Therefore $T = T_{M_1} \cup \cdots \cup T_{M_r} \cup T''_1 \cup \cdots \cup T''_t$ and any two one-column tableaux in the $\cup$-product are weakly separated.  
\end{proof}

\begin{example}
Let $T = \scalemath{0.6}{\begin{ytableau}
1 & 2 & 2 & 3 & 3 & 4 \\
3 & 5 & 5 & 6 & 8 & 9
\end{ytableau}}$. The factors of $T$ of the form $ \scalemath{0.6}{ \begin{ytableau}
a \\ \scalemath{0.6}{ a+1 }
\end{ytableau} }$ are $\scalemath{0.6}{ \begin{ytableau}
2 \\ 3
\end{ytableau}}$, $\scalemath{0.6}{\begin{ytableau}
4 \\ 5
\end{ytableau}}$. Removing these factors, we obtain $T' = \scalemath{0.6}{\begin{ytableau}
1 & 2 & 3 & 3 \\
5 & 6 & 8 & 9
\end{ytableau}}$. The corresponding $U_q(\widehat{\mathfrak{sl}_2})$-module is 
\begin{align*}
L(M_T) =L(Y_{1,-1}Y_{1,-3}^2Y_{1,-5}^4Y_{1,-7}^3Y_{1,-9}^2Y_{1,-11}^2Y_{1,-13}),    
\end{align*}
see Section \ref{subsec:dominant monomials and tableaux}. By taking all maximal strings of Kirillov-Reshetikhin modules, we have that 
\begin{align*}
\chi_q(L(M_T)) = \chi_q(L(M_1)) \chi_q(L(M_2)) \chi_q(L(M_3)) \chi_q(L(M_4)),    
\end{align*}
where 
\begin{align*}
M_1 = Y_{1,-1}Y_{1,-3}\cdots Y_{1,-13}, \ M_2 = Y_{1,-3}Y_{1,-5} \cdots Y_{1,-11}, \ M_3 = Y_{1,-5}, \ M_4 = Y_{1,-5}Y_{1,-7}.  
\end{align*}
The corresponding one-column tableaux are $\scalemath{0.6}{\begin{ytableau}
1 \\ 9
\end{ytableau}}$, $\scalemath{0.6}{\begin{ytableau}
2 \\ 8
\end{ytableau}}$, $\scalemath{0.6}{\begin{ytableau}
3 \\ 5
\end{ytableau}}$, $\scalemath{0.6}{\begin{ytableau}
3 \\ 6
\end{ytableau}}$, respectively. Therefore the unordered $6$-tuple of pairwise noncrossing one-column tableaux corresponding to $T$ is 
\begin{align*}
\scalemath{0.6}{ \left(  \begin{ytableau}
1 \\ 9
\end{ytableau}, \begin{ytableau}
2 \\ 8
\end{ytableau}, \begin{ytableau}
3 \\ 5
\end{ytableau}, \begin{ytableau}
3 \\ 6
\end{ytableau}, \begin{ytableau}
2 \\ 3
\end{ytableau}, \begin{ytableau}
4 \\ 5
\end{ytableau} \right) . }
\end{align*}

\end{example}

\begin{lemma} \label{lem:semistandard tableau and noncrossing tuple}
For every tableau $T \in {\rm SSYT}(k, [n])$ which has $m$ columns, there is a unique unordered $m$-tuple $(S_1, \ldots, S_m)$ of one-column tableaux which are pairwise noncrossing such that $T = S_1 \cup \cdots \cup S_m$. 
\end{lemma}

\begin{proof}
We prove by induction on $k$. The result is clearly true in the case of $k=1$. The case of $k=2$ is proved in Lemma \ref{lem:semistandard tableau and noncrossing tuple k=2}. 

Suppose that $k \ge 3$ and the result is true for $\SSYT(k', [n])$ for any $k' \le k-1$. Let $T \in {\rm SSYT}(k, [n])$. Let $T'$ be the sub-tableau of $T$ consisting of the first $k-1$ rows of $T$. By induction hypothesis, there there is a unique unordered $m$-tuple $S'=(S_1', \ldots, S_m')$ of one-column tableaux which are pairwise noncrossing such that $T' = S_1' \cup \cdots \cup S_m'$.  

Let $T''$ be the sub-tableau of $T$ consisting of the last $k-1$ rows of $T$. By induction hypothesis, there there is a unique unordered $m$-tuple $S''=(S_1'', \ldots, S_m'')$ of one-column tableaux which are pairwise noncrossing such that $T'' = S_1'' \cup \cdots \cup S_m''$.  

Let $T'''$ be the sub-tableau of $T$ consisting of the middle $k-2$ rows of $T$. By induction hypothesis, there there is a unique unordered $m$-tuple $S'''=(S_1''', \ldots, S_m''')$ of one-column tableaux which are pairwise noncrossing such that $T''' = S_1''' \cup \cdots \cup S_m'''$.

By Definition \ref{def:noncrossing pairs}, for any tuple of noncrossing tableaux, if we remove the first or the last entries of all tableaux in the tuple, the resulting tuple is still noncrossing. Therefore the tuple $S'''$ is obtained from $S'$ by removing the first entries, and the tuple $S'''$ is also obtained from $S''$ by removing the last entries. We can choose some ordering of $S', S'', S'''$ such that for each $j \in [m]$, $S_{j}'''=S_{j}' \cap S_{j}''$. Let $S_j = S_{j}' \cup S_{j}''$, $j \in [m]$. Then $T=S_1 \cup \cdots \cup S_m$ and $S_1, \ldots, S_m$ are pairwise noncrossing. 
\end{proof}

\begin{example}
Let $T= \scalemath{0.6}{ \begin{ytableau} 1 & 2 \\ 3 & 4 \\ 5 & 6 \\ 7 & 8 \end{ytableau}}$. The unique noncrossing 2-tuple of one-column tableaux corresponding to $T$ is $( \scalemath{0.6}{ \begin{ytableau} 1 \\ 4 \\ 5 \\ 8 \end{ytableau}}, \scalemath{0.6}{ \begin{ytableau} 2 \\ 3 \\ 6 \\ 7 \end{ytableau} } )$. 
\end{example}

\subsection{2-column prime tableaux} \label{subsec:2-column prime tableaux}

\begin{lemma} \label{lem:T1T2 non crossing and not weakly separated implies that if S1 S2 is T1 T2 then S1 S2 not weakly separated}
Suppose that $T_1, T_2$ are 1-column tableaux and they are noncrossing and not weakly separated. Then for any pair of 1-column tableaux $S_1, S_2$ such that $S_1 \cup S_2 = T_1 \cup T_2$, we have that $S_1, S_2$ are not weakly separated. 
\end{lemma}

\begin{proof}
Suppose that $T_1, T_2$ are 1-column tableaux and they are noncrossing and not weakly separated. By Lemma \ref{lem:semistandard tableau and noncrossing tuple}, for every pair of 1-column tableaux $S_1, S_2$ such that $S_1 \cup S_2 = T_1 \cup T_2$, either $\{S_1, S_2\} = \{T_1, T_2\}$ or $S_1, S_2$ are crossing. If $\{S_1, S_2\} = \{T_1, T_2\}$, then $S_1, S_2$ are not weakly separated. 

If $\{S_1, S_2\} \ne \{T_1, T_2\}$, then $S_1, S_2$ are crossing. If there are $1 \le a < b \le k$ such that the sub-tableau of $S_1$ consisting of the $a$th to $b$th rows of $S_1$ and the sub-tableau of $S_2$ consisting of the $a$th to $b$th rows of $S_2$ are not weakly separated, then $S_1$, $S_2$ are not weakly separated. 

Now suppose that for any $1 \le a < b \le k$, the sub-tableau of $S_1$ consisting of the $a$th to $b$th rows of $S_1$ and the sub-tableau of $S_2$ consisting of the $a$th to $b$th rows of $S_2$ are weakly separated. This contradicts the fact that $S_1$, $S_2$ are crossing.
\end{proof}
 
\begin{example}
Let $T_1= \scalemath{0.6}{ \begin{ytableau} 1 \\ 4 \\ 5 \\ 8 \end{ytableau} }, T_2 = \scalemath{0.6}{ \begin{ytableau} 2 \\ 3 \\ 6 \\ 7 \end{ytableau} }$. We have that $T_1, T_2$ are noncrossing and not weakly separated. All pairs of 1-column tableaux $S_1, S_2$ such that $S_1 \cup S_2 = T_1 \cup T_2$ are 
\begin{align*}
\scalemath{0.7}{
(\begin{ytableau} 1 \\ 3 \\ 5 \\ 7 \end{ytableau},  \begin{ytableau} 2 \\ 4 \\ 6 \\ 8 \end{ytableau}), 
(\begin{ytableau} 1 \\ 3 \\ 5 \\ 8 \end{ytableau},   \begin{ytableau} 2 \\ 4 \\ 6 \\ 7 \end{ytableau}), 
(\begin{ytableau} 1 \\ 3 \\ 6 \\ 7 \end{ytableau},   \begin{ytableau} 2 \\ 4 \\ 5 \\ 8 \end{ytableau}), 
(\begin{ytableau} 1 \\ 3 \\ 6 \\ 8 \end{ytableau},   \begin{ytableau} 2 \\ 4 \\ 5 \\ 7 \end{ytableau}), 
(\begin{ytableau} 1 \\ 4 \\ 5 \\ 7 \end{ytableau},  \begin{ytableau} 2 \\ 3 \\ 6 \\ 8 \end{ytableau}), 
(\begin{ytableau} 1 \\ 4 \\ 5 \\ 8 \end{ytableau},   \begin{ytableau} 2 \\ 3 \\ 6 \\ 7 \end{ytableau}), 
(\begin{ytableau} 1 \\ 4 \\ 6 \\ 7 \end{ytableau},   \begin{ytableau} 2 \\ 3 \\ 5 \\ 8 \end{ytableau}),
(\begin{ytableau} 1 \\ 4 \\ 6 \\ 8 \end{ytableau},   \begin{ytableau} 2 \\ 3 \\ 5 \\ 7 \end{ytableau}).  }
\end{align*} 
All of these pairs are not weakly separated. 
\end{example}

Every Pl\"{u}cker coordinate corresponds to a one-column tableau. Let $L(M_J)$, $L(M_{J'})$ be simple $U_q(\widehat{\mathfrak{sl}_k})$-modules corresponding to the Pl\"{u}cker coordinates $P_{J}, P_{J'}$ respectively. By Theorem 1.6 in \cite{OPS15}, $P_J$ and $P_{J'}$ are in the same cluster of the Grassmannian cluster algebra if and only if $J, J'$ are weakly separated. Recall that there is an isomorphism between $\CC[\Gr(k,n,\sim)]$ and $K_0(\mathcal{C}_{\ell})$, $n=k+\ell+1$, see Section \ref{sec:preliminaries}. Therefore for two $k$-subsets $J, J'$ whose entries are not consecutive sets, $L(M_J)$, $L(M_{J'})$ are in the same cluster of the cluster algebra $K_0(\mathcal{C}_{\ell})$ if and only if $J, J'$ are weakly separated. 

\begin{lemma} \label{lem:I J weakly separated implies that tensor product is simple}
Suppose that two $k$-element subsets $J, J'$ of $[n]$ are weakly separated. Then $L(M_J) \otimes L(M_{J'})$ is simple.
\end{lemma}

\begin{proof}
Suppose that $J, J'$ are weakly separated. Then $P_{J}, P_{J'}$ are in the same cluster of the Grassmannian cluster algebra, and $L(M_{J}), L(M_{J'})$ are in the same cluster of the cluster algebra $K_0(\mathcal{C}_{\ell})$. By the result in \cite{Qin17, KKOP21, KKOP24} that cluster monomials in $\mathcal{C}_{\ell}$ are simple modules, we have that $L(M_{J}) \otimes L(M_{J'})$ is simple.
\end{proof}

We conjecture that the converse of Lemma \ref{lem:I J weakly separated implies that tensor product is simple} is also true.
\begin{conjecture} \label{conjecture: tensor product is simple implies that I J weakly separated}
Let $J, J'$ be two $k$-element subsets of $[n]$. Suppose that $L(M_J) \otimes L(M_{J'})$ is simple, then $J, J'$ are weakly separated.
\end{conjecture}

Let $L(M)$ be a simple $U_q(\widehat{\mathfrak{sl}_k})$-module such that $T_M$ is a 2-column tableau. By Lemma \ref{lem:semistandard tableau and noncrossing tuple}, there is a unique pair $T_1, T_2$ of one-column tableaux $T_1, T_2$ such that $T_1, T_2$ are noncrossing and $T_M = T_1 \cup T_2$. Assume that Conjecture \ref{conjecture: tensor product is simple implies that I J weakly separated} is true. Then we have the following theorem.
\begin{theorem} \label{thm: Uqslkhat modules corresponding to 2 column tableaux}
Let $L(M)$ be a simple $U_q(\widehat{\mathfrak{sl}_k})$-module such that $T_M$ is a 2-column tableau. Then the module $L(M)$ is prime if and only if $T_1, T_2$ are not weakly separated, where $T_1, T_2$ are one-column tableaux such that $T_M = T_1 \cup T_2$ and $T_1, T_2$ are noncrossing.
\end{theorem}

\begin{proof}
Let $T_1, T_2$ be one-column tableaux such that $T_1, T_2$ are noncrossing and $T_M = T_1 \cup T_2$. Suppose that $T_1, T_2$ are weakly separated. By Lemma \ref{lem:I J weakly separated implies that tensor product is simple}, we have that $L(M_{T_1}) \otimes L(M_{T_2})$ is simple. It follows that $\chi_q(L(M)) = \chi_q(L(M_{T_1})) \chi_q(L(M_{T_2}))$. Therefore $L(M)$ is not prime.

Now suppose that $T_1, T_2$ are not weakly separated. By Lemma \ref{lem:T1T2 non crossing and not weakly separated implies that if S1 S2 is T1 T2 then S1 S2 not weakly separated}, for any pair $T_1', T_2'$ of 1-column tableaux such that $T_1 \cup T_2 = T_1' \cup T_2'$, we have that $T_1', T_2'$ are not weakly separated. Since we assume that Conjecture \ref{conjecture: tensor product is simple implies that I J weakly separated} is true, we have that $L(M_{T_1'}) \otimes L(M_{T_2'})$ is not simple. Therefore $\chi_q(L(M)) \ne \chi_q(L(M_{T_1'}))\chi_q(L(M_{T_2'}))$, for any pair of 1-column tableaux $T_1', T_2'$ such that $T_M = T_1' \cup T_2'$. Hence $L(M)$ is prime. 
\end{proof}

Assume that Conjecture \ref{conjecture: tensor product is simple implies that I J weakly separated} is true. Then we have Theorem \ref{thm: Uqslkhat modules corresponding to 2 column tableaux}. Let $T$ be a 2-column tableau. By Lemma \ref{lem:semistandard tableau and noncrossing tuple}, there is a unique pair $T_1, T_2$ of one-column tableaux $T_1, T_2$ such that $T_1, T_2$ are noncrossing and $T = T_1 \cup T_2$. Theorem \ref{thm: Uqslkhat modules corresponding to 2 column tableaux} implies that the 2-column tableau $T$ is prime if and only if $T_1, T_2$ are not weakly separated. 

Denote $\genfrac{(}{)}{0pt}{}{n}{a, b, c}=\frac{n!}{a!b!c!}$ and $I \Delta J = (I \setminus J) \cup (J \setminus I)$ for two sets $I, J$. 
\begin{proposition} \label{prop:number of 2-column prime tableaux}
For $k \le n/2$, the number of 2-column prime tableaux is $a_{k,n,2}-b_{k,n}$, where $a_{k,n,m}=\prod_{i=1}^k \prod_{j=1}^m \frac{n-i+j}{k+m-i-j+1}$ and $b_{k,n} = \binom{n}{k} + \sum_{j=1}^{k} j \genfrac{(}{)}{0pt}{}{n}{k-j,\, 2j,\, n-k-j}$.
\end{proposition}

\begin{proof}
The number of semistandard Young tableaux of rectangular shape with $k$ rows and with entries in $\{1,\ldots,n\}$ and with $m$ columns is $a_{k,n,m}$, see \cite{Sta99}. 

Assume that $k \le n/2$. If $I=J$, then $I, J$ are weakly separated and there are $\binom{n}{k}$ choices of $I=J$. Now assume that $I \ne J$. Denote $|I-J|=|J-I|=j$. Since $|I \cap J|=k-j$, $|I \Delta J|=2j$, there are $\genfrac{(}{)}{0pt}{}{n}{k-j, 2j, n-k-j}$ ways to fix the sets $I \cap J$ and $I \Delta J$. 

Since either $I \setminus J$ or $J \setminus I$ should be a segment of $s$ consecutive elements of the $2j$ elements in $I \Delta J$, there are $2j$ choices of $I - J$. Since the pair $(I, J)$ is unordered, there are $\frac{1}{2} \sum_{j=1}^{k} 2 j \genfrac{(}{)}{0pt}{}{n}{k-j, 2j, n-k-j}$ choices of weakly separated pairs $(I, J)$ (unordered) in the case of $I \ne J$. It follows that the number of unordered weakly separated pairs among all Pl\"{u}cker coordinates is $b_{k,n}$. 

Therefore the number of 2-column prime tableaux is $a_{k,n,2}-b_{k,n}$.
\end{proof}

\begin{remark}
It is conjectured in \cite{BBGL} that for $k \le n/2$, there are 
\begin{align*}
\sum_{r=3}^{k} \left( \frac{2r}{3} \cdot p_1(r) +  2r \cdot p_2(r) + 4r \cdot p_3(r) \right) \cdot \binom{n}{2r} \binom{n-2r}{k-r}
\end{align*} 
$2$-column cluster variables in $\CC[\Gr(k,n)]$, where $p_i(r)$ is the number of partitions $r=r_1+r_2+r_3$ such that $r_1,r_2,r_3 \in \ZZ_{\ge 1}$ and $|\{r_1,r_2,r_3\}|=i$. The number $a_{k,n,2}-b_{k,n}$ in Proposition \ref{prop:number of 2-column prime tableaux} includes prime tableaux which are not cluster variables. 
\end{remark}

\section{\texorpdfstring{2-column prime tableaux for $\CC[\Gr(4,8)]$ and $\CC[\Gr(5,10)]$}{2-column prime tableaux for CC[Gr(4,8)] and CC[Gr(5,10)]}}
\label{sec:2-column prime tableaux in Gr48 and Gr510}

In this section, we list prime tableaux for $\CC[\Gr(4,8)]$ and $\CC[\Gr(5,10)]$.

\subsection{Promotion of tableaux}

Promotion is an operator on the set of semistandard Young tableaux defined in terms of ``jeu de taquin'' sliding moves \cite{Sch63, Sch72, Sch77}. Gansner \cite{Ga80} proved that promotion can also be described using using Bender-Knuth involutions \cite{BK}. 
In this paper, we only need the promotion operator on $\SSYT(k, [n])$. 

The $i$th ($i \in [n]$) Bender-Knuth involution ${\rm BK}_i: \SSYT(k, [n]) \to \SSYT(k, [n])$, is defined by the following procedure: for $i$ and $i+1$ which are not in the same column, we replace $i$ by $i+1$ and replace $i+1$ by $i$, and then reorder $i$, $i+1$ in each row such that the resulting tableau is semistandard. The promotion ${\rm pr}(T)$ of $T$ is defined by 
\begin{align*}
{\rm pr}(T) = {\rm BK}_1 \circ \cdots \circ {\rm BK}_{n-1}(T),    
\end{align*}
see also Definition A.3 in \cite{Hop20}.  

\subsection{\texorpdfstring{2-column prime tableaux for $\CC[\Gr(4,8)]$}{2-column prime tableaux for CC[Gr(4,8)]}}

There are totally $122$ prime tableaux for $\CC[\Gr(4,8)]$. Two of them are non-real:
\begin{align*}
\scalemath{0.7}{ \begin{ytableau}
 1&2\\3&4\\5&6\\7&8
\end{ytableau}, \quad \begin{ytableau}
1&3\\2&5\\4&7\\6&8
\end{ytableau}. }
\end{align*}
They corresponds to the following prime non-real modules respectively:
\begin{align*}
L(Y_{3, 1}Y_{2, -2}Y_{3, -1}Y_{1, -5}Y_{2, -4}Y_{1, -7}), \ L(Y_{2, 0}Y_{1, -3}Y_{3, -3}Y_{2, -6}).
\end{align*}

Up to promotion, there are $15$ real prime tableaux:
\begin{align*}
\scalemath{0.66}{ 
\begin{ytableau}
1 & 1 \\
2 & 4 \\
3 & 6 \\
5 & 7
\end{ytableau},  \ \begin{ytableau}
1 & 1 \\
2 & 4 \\
3 & 6 \\
5 & 8
\end{ytableau},  \ \begin{ytableau}
1 & 1 \\
2 & 4 \\
3 & 7 \\
5 & 8
\end{ytableau},  \ \begin{ytableau}
1 & 2 \\
2 & 4 \\
3 & 6 \\
5 & 7
\end{ytableau},  \ \begin{ytableau}
1 & 2 \\
2 & 4 \\
3 & 6 \\
5 & 8
\end{ytableau},  \ \begin{ytableau}
1 & 2 \\
2 & 4 \\
3 & 7 \\
5 & 8
\end{ytableau},  \ \begin{ytableau}
1 & 3 \\
2 & 4 \\
3 & 6 \\
5 & 8
\end{ytableau},  \ \begin{ytableau}
1 & 3 \\
2 & 4 \\
3 & 7 \\
5 & 8
\end{ytableau},  \ \begin{ytableau}
1 & 4 \\
2 & 6 \\
3 & 7 \\
5 & 8
\end{ytableau},  \ \begin{ytableau}
1 & 2 \\
2 & 4 \\
3 & 7 \\
6 & 8
\end{ytableau},  \ \begin{ytableau}
1 & 2 \\
2 & 5 \\
3 & 7 \\
6 & 8
\end{ytableau},  \ \begin{ytableau}
1 & 3 \\
2 & 4 \\
3 & 7 \\
6 & 8
\end{ytableau},  \ \begin{ytableau}
1 & 3 \\
2 & 5 \\
3 & 7 \\
6 & 8
\end{ytableau},  \ \begin{ytableau}
1 & 1 \\
2 & 3 \\
4 & 5 \\
6 & 7
\end{ytableau},  \ \begin{ytableau}
1 & 2 \\
2 & 5 \\
4 & 7 \\
6 & 8
\end{ytableau}. }
\end{align*}
They corresponds to the following prime real modules respectively:
\begin{align*}
& L(Y_{1, -1}Y_{3, 1}Y_{3, -1}Y_{2, -4}), \
L( Y_{1, -1}Y_{3, 1}Y_{3, -1}Y_{2, -4}Y_{1, -7}), \
L( Y_{1, -1}Y_{3, 1}Y_{3, -1}Y_{2, -4}Y_{2, -6}), \\
& L( Y_{1, -1}Y_{3, -1}Y_{2, -4}), \ L( Y_{1, -1}Y_{3, -1}Y_{2, -4}Y_{1, -7}), \
L( Y_{1, -1}Y_{3, -1}Y_{2, -4}Y_{2, -6}), \ 
L( Y_{1, -1}Y_{2, -4}Y_{1, -7}), 
\end{align*}
\begin{align*}
& 
L( Y_{1, -1}Y_{2, -4}Y_{2, -6}), \ 
L( Y_{1, -1}Y_{3, -5}), \
L( Y_{1, -1}Y_{1, -3}Y_{3, -1}Y_{2, -4}Y_{2, -6}), \\
& 
L( Y_{1, -1}Y_{1, -3}Y_{3, -1}Y_{3, -3}Y_{2, -6}), \
L( Y_{1, -1}Y_{1, -3}Y_{2, -4}Y_{2, -6}), \
L( Y_{1, -1}Y_{1, -3}Y_{3, -3}Y_{2, -6}), \\
&
L( Y_{2, 0}Y_{3, 1}Y_{1, -3}Y_{2, -2}Y_{1, -5}), \
L( Y_{2, 0}Y_{1, -3}Y_{3, -1}Y_{3, -3}Y_{2, -6}).
\end{align*}

\subsection{\texorpdfstring{2-column prime tableaux for $\CC[\Gr(5,10)]$}{2-column prime tableaux for CC[Gr(5,10)]}}

There are totally $197$ prime non-real tableaux for $\CC[\Gr(5,10)]$. Up to promotion, there are $21$ prime non-real tableaux:
\begin{align*}
& \scalemath{0.7}{ \begin{ytableau}
1 & 1 \\
2 & 3 \\
4 & 5 \\
6 & 7 \\
8 & 9
\end{ytableau},  \ \begin{ytableau}
1 & 1 \\
2 & 3 \\
4 & 5 \\
6 & 7 \\
8 & 10
\end{ytableau},  \ \begin{ytableau}
1 & 1 \\
2 & 3 \\
4 & 5 \\
6 & 7 \\
9 & 10
\end{ytableau},  \ \begin{ytableau}
1 & 1 \\
2 & 3 \\
4 & 5 \\
6 & 8 \\
9 & 10
\end{ytableau},  \ \begin{ytableau}
1 & 1 \\
2 & 3 \\
4 & 5 \\
7 & 8 \\
9 & 10
\end{ytableau},  \ \begin{ytableau}
1 & 1 \\
2 & 3 \\
4 & 6 \\
7 & 8 \\
9 & 10
\end{ytableau},  \ \begin{ytableau}
1 & 1 \\
2 & 3 \\
5 & 6 \\
7 & 8 \\
9 & 10
\end{ytableau},  \ \begin{ytableau}
1 & 1 \\
2 & 4 \\
3 & 6 \\
5 & 8 \\
7 & 9
\end{ytableau},  \ \begin{ytableau}
1 & 1 \\
2 & 4 \\
3 & 6 \\
5 & 8 \\
7 & 10
\end{ytableau},  \ \begin{ytableau}
1 & 1 \\
2 & 4 \\
3 & 6 \\
5 & 9 \\
7 & 10
\end{ytableau}, }  \\
& \scalemath{0.7}{ \begin{ytableau}
1 & 1 \\
2 & 4 \\
3 & 6 \\
5 & 9 \\
8 & 10
\end{ytableau},  \ \begin{ytableau}
1 & 1 \\
2 & 4 \\
3 & 7 \\
5 & 9 \\
8 & 10
\end{ytableau},  \ \begin{ytableau}
1 & 1 \\
2 & 4 \\
3 & 7 \\
6 & 9 \\
8 & 10
\end{ytableau},  \ \begin{ytableau}
1 & 1 \\
2 & 4 \\
5 & 6 \\
7 & 8 \\
9 & 10
\end{ytableau},  \ \begin{ytableau}
1 & 1 \\
2 & 5 \\
3 & 7 \\
6 & 9 \\
8 & 10
\end{ytableau},  \ \begin{ytableau}
1 & 1 \\
2 & 5 \\
4 & 7 \\
6 & 9 \\
8 & 10
\end{ytableau},  \ \begin{ytableau}
1 & 1 \\
3 & 4 \\
5 & 6 \\
7 & 8 \\
9 & 10
\end{ytableau},  \ \begin{ytableau}
1 & 1 \\
3 & 5 \\
4 & 7 \\
6 & 9 \\
8 & 10
\end{ytableau},  \ \begin{ytableau}
1 & 2 \\
3 & 4 \\
5 & 6 \\
7 & 8 \\
9 & 10
\end{ytableau},  \ \begin{ytableau}
1 & 2 \\
3 & 4 \\
5 & 6 \\
7 & 9 \\
8 & 10
\end{ytableau},  \ \begin{ytableau}
1 & 2 \\
3 & 4 \\
5 & 7 \\
6 & 9 \\
8 & 10
\end{ytableau}. }
\end{align*}
They corresponds to the following prime non-real modules respectively:
\begin{align*}
& L(Y_{3, 1}Y_{4, 2}Y_{2, -2}Y_{3, -1}Y_{1, -5}Y_{2, -4}Y_{1, -7}), \ 
L(Y_{3, 1}Y_{4, 2}Y_{2, -2}Y_{3, -1}Y_{1, -5}Y_{2, -4}Y_{1, -7}Y_{1, -9}), \\
&
L(Y_{3, 1}Y_{4, 2}Y_{2, -2}Y_{3, -1}Y_{1, -5}Y_{2, -4}Y_{1, -7}Y_{1, -7}Y_{1, -9}), \
L(Y_{3, 1}Y_{4, 2}Y_{2, -2}Y_{3, -1}Y_{1, -5}Y_{2, -4}Y_{1, -7}Y_{2, -6}Y_{1, -9}), 
\end{align*}
\begin{align*}
& 
L(Y_{3, 1}Y_{4, 2}Y_{2, -2}Y_{3, -1}Y_{2, -4}Y_{2, -4}Y_{1, -7}Y_{2, -6}Y_{1, -9}), \
L(Y_{3, 1}Y_{4, 2}Y_{2, -2}Y_{3, -1}Y_{2, -4}Y_{3, -3}Y_{1, -7}Y_{2, -6}Y_{1, -9}), \\
& 
L(Y_{3, 1}Y_{4, 2}Y_{3, -1}Y_{3, -1}Y_{2, -4}Y_{3, -3}Y_{1, -7}Y_{2, -6}Y_{1, -9}), \
L(Y_{2, 0}Y_{4, 2}Y_{1, -3}Y_{4, 0}Y_{3, -3}Y_{2, -6}), 
\end{align*}
\begin{align*}
& 
L(Y_{2, 0}Y_{4, 2}Y_{1, -3}Y_{4, 0}Y_{3, -3}Y_{2, -6}Y_{1, -9}), \
L(Y_{2, 0}Y_{4, 2}Y_{1, -3}Y_{4, 0}Y_{3, -3}Y_{2, -6}Y_{2, -8}), \\ 
& 
L(Y_{2, 0}Y_{4, 2}Y_{1, -3}Y_{4, 0}Y_{1, -5}Y_{3, -3}Y_{2, -6}Y_{2, -8}), \
L(Y_{2, 0}Y_{4, 2}Y_{1, -3}Y_{4, 0}Y_{1, -5}Y_{3, -3}Y_{3, -5}Y_{2, -8}), 
\end{align*}
\begin{align*}
& 
L(Y_{2, 0}Y_{4, 2}Y_{2, -2}Y_{4, 0}Y_{1, -5}Y_{3, -3}Y_{3, -5}Y_{2, -8}), \
L(Y_{3, 1}Y_{4, 2}Y_{3, -1}Y_{4, 0}Y_{2, -4}Y_{3, -3}Y_{1, -7}Y_{2, -6}Y_{1, -9}), \\ 
& 
L(Y_{2, 0}Y_{4, 2}Y_{2, -2}Y_{4, 0}Y_{1, -5}Y_{4, -2}Y_{3, -5}Y_{2, -8}), \
L(Y_{3, 1}Y_{4, 2}Y_{2, -2}Y_{4, 0}Y_{1, -5}Y_{4, -2}Y_{3, -5}Y_{2, -8}), 
\end{align*}
\begin{align*}
& 
L(Y_{4, 2}^2 Y_{3, -1}Y_{4, 0}Y_{2, -4}Y_{3, -3}Y_{1, -7}Y_{2, -6}Y_{1, -9}), \
L(Y_{4, 2}^2 Y_{2, -2}Y_{4, 0}Y_{1, -5}Y_{4, -2}Y_{3, -5}Y_{2, -8}), \\ 
&
L(Y_{4, 2}Y_{3, -1}Y_{4, 0}Y_{2, -4}Y_{3, -3}Y_{1, -7}Y_{2, -6}Y_{1, -9}), \
L(Y_{4, 2}Y_{3, -1}Y_{4, 0}Y_{2, -4}Y_{3, -3}Y_{2, -6}Y_{2, -8}), \\
&
L(Y_{4, 2}Y_{3, -1}Y_{4, 0}Y_{1, -5}Y_{3, -3}Y_{3, -5}Y_{2, -8}).
\end{align*}

There are totally $3260$ prime real tableaux for $\CC[\Gr(5,10)]$. Up to promotion, there are $326$ prime real tableaux. Ten of these tableaux are:
\begin{align*}
\scalemath{0.7}{
\begin{ytableau}
1 & 1 \\
2 & 2 \\
3 & 5 \\
4 & 7 \\
6 & 8
\end{ytableau}, \ \begin{ytableau}
1 & 1 \\
2 & 2 \\
3 & 5 \\
4 & 7 \\
6 & 9
\end{ytableau}, \ \begin{ytableau}
1 & 1 \\
2 & 2 \\
3 & 5 \\
4 & 7 \\
6 & 10
\end{ytableau}, \ \begin{ytableau}
1 & 1 \\
2 & 2 \\
3 & 5 \\
4 & 8 \\
6 & 9
\end{ytableau}, \ \begin{ytableau}
1 & 1 \\
2 & 2 \\
3 & 5 \\
4 & 8 \\
6 & 10
\end{ytableau}, \ \begin{ytableau}
1 & 1 \\
2 & 2 \\
3 & 5 \\
4 & 9 \\
6 & 10
\end{ytableau}, \ \begin{ytableau}
1 & 1 \\
2 & 3 \\
3 & 5 \\
4 & 7 \\
6 & 8
\end{ytableau}, \ \begin{ytableau}
1 & 1 \\
2 & 3 \\
3 & 5 \\
4 & 7 \\
6 & 9
\end{ytableau}, \ \begin{ytableau}
1 & 1 \\
2 & 3 \\
3 & 5 \\
4 & 7 \\
6 & 10
\end{ytableau}, \ \begin{ytableau}
1 & 1 \\
2 & 3 \\
3 & 5 \\
4 & 8 \\
6 & 9
\end{ytableau}.
}
\end{align*}
They correspond to the following real prime modules respectively:
\begin{align*}
& L(Y_{1, -1}Y_{3, 1}Y_{3, -1}Y_{2, -4}), \ L( Y_{1, -1}Y_{3, 1}Y_{3, -1}Y_{2, -4}Y_{1, -7}), \ L( Y_{1, -1}Y_{3, 1}Y_{3, -1}Y_{2, -4}Y_{1, -7}Y_{1, -9}), \\ 
& L( Y_{1, -1}Y_{3, 1}Y_{3, -1}Y_{2, -4}Y_{2, -6}), \ L( Y_{1, -1}Y_{3, 1}Y_{3, -1}Y_{2, -4}Y_{2, -6}Y_{1, -9}), 
\end{align*}
\begin{align*}
& L( Y_{1, -1}Y_{3, 1}Y_{3, -1}Y_{2, -4}Y_{2, -6}Y_{2, -8}), \
L( Y_{1, -1}Y_{4, 2}Y_{3, -1}Y_{2, -4}), \ L( Y_{1, -1}Y_{4, 2}Y_{3, -1}Y_{2, -4}Y_{1, -7}), \\
& L( Y_{1, -1}Y_{4, 2}Y_{3, -1}Y_{2, -4}Y_{1, -7}Y_{1, -9}), \ L( Y_{1, -1}Y_{4, 2}Y_{3, -1}Y_{2, -4}Y_{2, -6}).
\end{align*}

\section{Prime modules corresponding to tableaux with two or more columns} \label{sec:Prime modules corresponding to tableaux with two or more columns}

In this section, we give a necessary condition for a tableau to be prime. We also give a conjecture that every pairwise noncrossing but not weakly separated collection of 1-column 
semistandard tableaux give a prime tableau.
 
\begin{conjecture} \label{conj:pairwise noncrossing and not weakly separated collections induce prime modules}
Let $T \in {\rm SSYT}(k, [n])$ and let $(S_1, \ldots, S_m)$ be the unique unordered $m$-tuple $(S_1, \ldots, S_m)$ of one-column tableaux which are pairwise noncrossing such that $T = S_1 \cup \cdots \cup S_m$. If for every $i \ne j$, $S_i, S_j$ are not weakly separated, then $T$ is prime.
\end{conjecture}

Conjecture \ref{conj:pairwise noncrossing and not weakly separated collections induce prime modules} gives an explicit description of the highest $l$-weights of a very large family of prime $U_q(\widehat{\mathfrak{sl}_k})$-modules.

Note that the condition in Conjecture \ref{conj:pairwise noncrossing and not weakly separated collections induce prime modules} is a sufficient condition but not a necessary condition. For example, in the case of $\Gr(3,8)$, the eight tableaux in (\ref{eq:the eight talbeaux which is not in N1}) are prime, see \cite{CDHHHL, RSV21}. But they do not satisfy the condition in Conjecture \ref{conj:pairwise noncrossing and not weakly separated collections induce prime modules}. For example, the unique $3$-tuple of pairwise non-crossing tableaux corresponding to the first tableau is
\begin{align*}
\scalemath{0.7}{ (\begin{ytableau}
 1\\6\\7
\end{ytableau}, \quad \begin{ytableau}
 2\\5\\7
\end{ytableau}, \quad \begin{ytableau}
 3\\4\\8
\end{ytableau}).  }
\end{align*}
The first two 1-column tableaux are weakly separated. 

\begin{align} \label{eq:the eight talbeaux which is not in N1}
& \scalemath{0.7}{ \begin{ytableau}
 1 & 2 & 3 \\
 6 & 5 & 4 \\
 7 & 7 & 8 \\
\end{ytableau}, \
\begin{ytableau}
 1 & 2 & 5 \\
 4 & 3 & 7 \\
 6 & 6 & 8 \\
\end{ytableau}, \
\begin{ytableau}
 1 & 3 & 4 \\
 2 & 6 & 5 \\
 5 & 7 & 8 \\
\end{ytableau}, \
\begin{ytableau}
 1 & 2 & 3 \\
 5 & 4 & 4 \\
 6 & 8 & 7 \\
\end{ytableau}, \ 
\begin{ytableau}
 1 & 2 & 4 \\
 3 & 3 & 7 \\
 6 & 5 & 8 \\
\end{ytableau}, \
\begin{ytableau}
 1 & 2 & 3 \\
 2 & 6 & 5 \\
 4 & 7 & 8 \\
\end{ytableau}, \
\begin{ytableau}
 1 & 1 & 2 \\
 4 & 5 & 3 \\
 7 & 6 & 8 \\
\end{ytableau}, \
\begin{ytableau}
 1 & 3 & 4 \\
 2 & 7 & 6 \\
 5 & 8 & 8 \\
\end{ytableau}. }
\end{align}

In the case of $r=2$, Conjecture \ref{conj:pairwise noncrossing and not weakly separated collections induce prime modules} is proved in Section \ref{sec:prime modules corresponding to 2-column tableaux}. 

\begin{example}
In the case of $\Gr(3,9)$, there are $3$ pairwise noncrossing and not weakly separated $3$-tuples:
\begin{align*}
\scalemath{0.7}{  ( \begin{ytableau}
1 \\ 6 \\ 7
\end{ytableau}, \begin{ytableau}
    2 \\ 5 \\ 8
\end{ytableau}, \begin{ytableau}
3 \\ 4 \\ 9
\end{ytableau} ), \ ( \begin{ytableau}
1 \\ 4 \\ 6
\end{ytableau}, \begin{ytableau}
    2 \\ 3 \\ 7
\end{ytableau}, \begin{ytableau}
5 \\ 8 \\ 9
\end{ytableau} ), \ ( \begin{ytableau}
1 \\ 2 \\ 5
\end{ytableau}, \begin{ytableau}
    3 \\ 7 \\ 8
\end{ytableau}, \begin{ytableau}
4 \\ 6 \\ 9
\end{ytableau} ).  }
\end{align*}

They correspond to $3$ prime tableaux in $\SSYT(3, [9])$:
\begin{align*}
\scalemath{0.7}{ \begin{ytableau}
    1 & 2 & 3 \\
 4 & 5 & 6 \\
 7 & 8 & 9
\end{ytableau}, \
 \begin{ytableau}
      1 & 2 & 5 \\
 3 & 4 & 8 \\
 6 & 7 & 9
 \end{ytableau}, \ 
\begin{ytableau}
    1 & 3 & 4 \\
 2 & 6 & 7 \\
 5 & 8 & 9
\end{ytableau}. }
\end{align*}

The corresponding prime non-real modules are
\begin{align*}
& L(Y_{2, 0}Y_{2, -2}^2 Y_{1, -5}Y_{2, -4}^2 Y_{1, -7}^2 Y_{2, -6}Y_{1, -9}^2 Y_{1, -11}), \\
& L(Y_{2, 0}Y_{1, -3}Y_{2, -2}Y_{1, -5}^2 Y_{1, -7}Y_{2, -8}Y_{2, -10}), \\
& L(Y_{1, -1}Y_{1, -3}Y_{2, -4}Y_{2, -6}^2 Y_{1, -9}Y_{2, -8}Y_{1, -11}).
\end{align*}

\end{example}

\end{document}